\documentclass{amsproc}
\usepackage{amsmath,amsfonts,amssymb,amsthm,amscd,mathrsfs,mathdots,bbm,color}
\usepackage{geometry}
\title{Moments of discrete classical $q$-orthogonal polynomial ensembles}
\author[P. Cohen]{Philip Cohen}
\address{Mathematical Institute, University of Oxford, Oxford, OX2 6GG, UK.}
\email{cohenp@maths.ox.ac.uk}
\thanks{Research supported by ERC Advanced Grant 669306.}

\theoremstyle{plain}
\newtheorem{thm}{Theorem}[section]
\newtheorem{lem}[thm]{Lemma}
\newtheorem{rmk}[thm]{Remark}
\newtheorem{cor}[thm]{Corollary}

\begin{document}

\begin{abstract}
We consider some discrete $q$-analogues of the classical continuous orthogonal polynomial ensembles.  Building on results due to Morozov, Popolitov and Shakirov, we find representations for the moments of the discrete $q$-Hermite and discrete $q$-Laguerre ensembles in terms of basic hypergeometric series.  We find that when the number of particles is suitably randomised,  the moments may be represented as basic hypergeometric orthogonal polynomials, with corresponding three-term recurrences in $k$, the order of the moments.
\end{abstract}

\maketitle

\section{Introduction}
An orthogonal polynomial ensemble is a probability measure on $\mathbb{R}^N$ given by
\begin{equation} \label{eq:dQ}
\mathrm{d}Q_N(\mathbf{x})=\frac{1}{Z_N} \Delta_N(\mathbf{x})^2 \prod_{j=1}^N \mathrm{d} \mu(x_j),
\end{equation}
where $\mathbf{x} = (x_1, \dots, x_N)$, is an $N$-tuple, $\mu$ is a probability measure on the real line having all moments,
\begin{equation}
\Delta_N(\mathbf{x})=\det_{1\leq i,j\leq N}\left(x_i^{j-1}\right)=\prod_{i<j}(x_j-x_i),
\end{equation}
which is known as the Vandermonde determinant, and $Z_N$ is a normalisation constant,  ensuring that the probability measure integrates to 1. The measure~\eqref{eq:dQ} can be conveniently analysed by using the so-called \emph{orthogonal polynomial method} pioneered by Mehta~\cite{Mehta60} in the study of random matrices. Denote by $p_n(x)$, $n\in \mathbb{N}$, the orthonormal polynomials with respect to the measure $\mu$. Then, a standard calculation shows that
\begin{equation}
\mathrm{d} Q_N(\mathbf{x})=\frac{1}{N!}\det_{1\leq i,j\leq N}K(x_i,x_j)\prod_{j=1}^N\mathrm{d}\mu(x_j),
\label{eq:dQ_K}
\end{equation}
where $K(x,y)=\sum_{n=0}^{N-1}p_n(x)p_n(y)$. In fact, all the marginals~\eqref{eq:dQ_K} can be expressed as determinants of  the correlation kernel $K(x,y)$. In particular, for suitable functions $f$,
\begin{equation}
\int_{\mathbb{R}^N}\left(\frac{1}{N}\sum_{n=1}^Nf(x_n)\right)\mathrm{d} Q_N(\mathbf{x})=\int_{\mathbb{R}}f(x)\mathrm{d}\rho_N(x),
\end{equation}
where the normalised one-point function $\rho_N$ is the probability measure
\begin{equation} \label{eq:onept}
\mathrm{d} \rho_N(x)=\frac{1}{N}\sum_{n=0}^{N-1}p_n(x)^2\mathrm{d} \mu(x).
\end{equation}
Therefore the study of the probability measure $Q_N$ for some reference measure $\mu$ amounts to understanding properties of the associated orthogonal polynomials $p_n(x)$. 

One of the most studied quantities on orthogonal polynomial ensembles are the so-called \emph{moments}. Let $\mathbf{X}=(X_1,\dots,X_N)$ be distributed according to the probability measure $Q_N = Q$. Then, the $k^{th}$ moment of $\mathbf{X}$ is the expectation of the $k^{th}$ power sum
\begin{equation} \label{eq:mom}
\mathbb{E}_Q \left[ \frac{1}{N} \sum_{n=1}^N X_n^k \right] 
=\int_{\mathbb{R}} x^k \mathrm{d} \rho_N(x),\quad k\in \mathbb{N}.
\end{equation}
In ensembles with random matrix interpretations, where the $X_i$'s can be thought of as eigenvalues of a matrix $M_N$, clearly this is the same as calculating the expectation of powers of traces,
\begin{equation}
\mathbb{E}_Q \left[ \frac{1}{N} \sum_{n=1}^N X_n^k \right] =  \frac{1}{N} \mathbb{E}_Q \left[ \text{tr} M_N^k \right].
\end{equation}
For the classical orthogonal polynomials, the corresponding ensembles arise naturally as eigenvalue distributions in random matrix theory.  These are known as the Gaussian (GUE), Laguerre (LUE) and Jacobi (JUE) unitary ensembles, with associated Hermite, Laguerre and Jacobi polynomials \cite{Mehta67}. Random matrices of this type have been studied extensively,  as they arise naturally in many settings. Wishart used random matrices for statistical analysis in the 1930s,  and developed the Complex Wishart distribution (which is equivalent to the LUE mentioned above) in the 1950s. Meanwhile Wigner applied similar models to study nuclear physics in the 1950s and 1960s.

The moments of the Gaussian, Laguerre and Jacobi unitary ensembles satisfy three-term recurrences in $k$ (the order of the moment). These remarkable recursions were first discovered by Harer and Zagier~\cite{Harer86} for the GUE, and by Haagerup and Thorbj\o rnsen for the LUE~\cite{Haagerup03} (the extension of the Haagerup-Thorbj\o rnsen recursion to moments of the inverse LUE was examined later in~\cite{Cunden16b}). Recursions in $k$ for moments of the JUE were obtained by Ledoux~\cite{Ledoux04}, and recast as three-term recurrences in~\cite{Cunden19}.  

Recent studies by Cunden, Mezzadri, O'Connell and Simm~\cite{Cunden19} clarified the origin of these three-term recurrences in $k$ for moments of the classical ensembles. In fact, the Harer-Zagier, Haagerup-Thorbj\o rnsen and Ledoux recursions can be interpreted as second order difference equation in the variable $k$ (discrete Sturm Liouville problems), and their solutions are hypergeometric orthogonal polynomials in $k$ belonging to the Askey scheme~\cite{Ismail05, Koekoek10}. The polynomial structure, orthogonality relation and hypergeometric representation of the moments as function of $k$ explain several nontrivial symmetries and provide new results. For instance, by duality, moments of classical random matrices, if suitably normalised can be viewed as hypergeometric  polynomials in the variable $(N-1)$ as well; therefore, they also satisfy three-term recurrences in $N$, the nature of which is different from the general topological recursion \cite{BG, CE, Eynard15,Eynard16, Guionnet15}.

In a recent work \cite{Cohen19}, some classical discrete ensembles (Charlier, Meixner and Krawtchouk) were considered, building on some work of Ledoux \cite{Ledoux05}, and hypergeometric formulas for the moments were found. Furthermore, when the number of particles is randomised according to certain distributions, the moments of these new ensembles are hypergeometric orthogonal polynomials, and therefore satisfy three-term recurrences in $k$.

In this paper, we consider some discrete $q$-analogues of the classical continuous orthogonal polynomial ensembles.  Building on results due to Morozov, Popolitov and Shakirov \cite{Morozov20, MPSqt} we find representations for the moments of the discrete $q$-Hermite and discrete $q$-Laguerre ensembles in terms of basic hypergeometric series.  We find that when the number of particles is suitably randomised,  the moments may be represented as basic hypergeometric orthogonal polynomials, with corresponding three-term recurrences in $k$, the order of the moments.

Finally, we remark that this paper is based on results from the author’s PhD thesis \cite{CohenThesis}.  After the final version of the thesis was submitted, a related independent work by Forrester, Li, Shen and Yu \cite{Forrester21} appeared on arXiv which has some overlap but is mostly complementary to the present work.

The outline of this paper is as follows. In Section \ref{se:Class}, we recall results concerning the moments of the classical GUE and LUE. In particular,  the generating functions and three-term recurrences satisfied by the moments.  In Section \ref{se:Pre} we introduce preliminary definitions and notations for the $q$-analogues, basic hypergeometric series, and the Jackson $q$-integral. In Section \ref{se:qHerm},  we define the discrete $q$-Hermite ensemble and, building on results of Morozov, Popolitov and Shakirov \cite{Morozov20}, we obtain a formula for the moments in terms of basic hypergeometric functions. We find that when the number of particles is suitably randomised, the moments are $q$-Hahn polynomials and thus satisfy three-term recurrences. Finally in Section \ref{se:qLag},  a new formula for the moments of the $q$-Laguerre ensemble is derived using results of Morozov, Popolitov and Shakirov \cite{MPSqt}. As in the $q$-Hermite case, when the number of particles is suitably randomised, the moments are Big $q$-Jacobi polynomials, and again satisfy three-term recurrences.

\section{Classical results} \label{se:Class}
In this section, we state some results about the classical continuous ensembles, namely the GUE (or Hermite ensemble, since the associated polynomials are the Hermite polynomials), and the LUE (or Laguerre ensemble).
\subsection{Hermite ensemble}
Consider an $N \times N$ Hermitian matrix $X=(X_{i,j})$ with independent random Gaussian entries, with standard Real Gaussian random variables on the diagonal, and Complex Gaussians below the diagonal. That is, a matrix chosen according to the following probability measure on $N \times N$ Hermitian matrices
\begin{equation}
P_N(X) \mathrm{d} X = \prod_{i=1}^N \frac{1}{\sqrt{2 \pi}} e^{-x_{ii}^2 / 2} \mathrm{d} x_{ii}
\prod_{i < j} \frac{1}{\pi} e^{-|x_{ij}|^2} \mathrm{d} \Re x_{ij} \mathrm{d} \Im x_{ij}.
\end{equation}
Matrices of this form are said to belong to the \emph{Gaussian Unitary Ensemble}, or GUE, so called because the measure is invariant under conjugation by unitary matrices. Then it is a well-known result in random matrix theory that the distribution of the eigenvalues of $X$ is a probability measure on $\mathbb{R}^N$ 
\begin{equation}
\mathrm{d}Q_N(\lambda_1, \dots, \lambda_N) = \frac{1}{ \prod_{i=1}^N i!} \Delta_N(\lambda_1, \dots, \lambda_N)^2 \prod_{j=1}^N \frac{1}{\sqrt{2 \pi}} e^{-\lambda_j^2 / 2} \mathrm{d} \lambda_j,
\end{equation}
which is exactly the form of \eqref{eq:dQ} with the reference measure $\mu$ being a standard Gaussian.  The Hermite polynomials are orthogonal with respect to the standard Gaussian measure on the real line
\begin{equation}
\mathrm{d} \mu(x) = \frac{1}{\sqrt{2 \pi}} e^{-x^2/2} \mathrm{d} x,
\end{equation}
and the $n^{th}$ normalised Hermite polynomial $h_n(x)$ can be written as \cite[9.15]{Koekoek10}
\begin{equation} \label{eq:Hermite}
h_n(x) = \frac{ x^n}{\sqrt{n!}} {}_2F_0 \left( \begin{matrix} -\frac{n}{2}, \frac{1-n}{2}  \\ - \end{matrix} ; -\frac{2}{x^2} \right). 
\end{equation} 
We denote the $2k^{th}$ moment of the Hermite ensemble by
\begin{equation}
Q_k(n) := \frac{\mathbb{E}[\text{tr} X_n^{2k} ]}{n}.
\end{equation}
Then the following results about the moments are known.
\begin{thm}[Harer-Zagier \cite{Harer86}]
The even moments of the Hermite ensemble $Q_k(n)$ satisfy the following three-term recurrence in $k$:
\begin{equation} \label{eq:HZ_rec}
(k+2) Q_{k+1}(n) = 2n(2k+1) Q_k(n) + k(2k+1)(2k-1) Q_{k-1}(n).
\end{equation}
\end{thm}
\begin{thm}[Harer-Zagier \cite{Harer86}] \label{th:HZ_gf}
The even moments of the Hermite ensemble $Q_k(n)$ satisfy the generating function in $n$ for fixed $k$,
\begin{equation} \label{eq:GUE_gf}
\sum_{n \geq 1} \frac{ Q_k(n)}{(2k-1)!!} z^n = \frac{z}{1-z^2} \left( \frac{1+z}{1-z} \right)^{k+1}
\end{equation}
\end{thm}
\begin{thm}[Witte-Forrester \cite{Witte14}]
The $2k^{th}$ moment of the Hermite ensemble one-point function is given by the following hypergeometric series.
\begin{equation}
Q_k(n) = (2k-1)!! {}_2F_1 \left( \begin{matrix} -k ,  1-n  \\ 2\end{matrix} ; 2 \right),
\end{equation}
and we note that the odd moments are 0, by symmetry.
\end{thm}
The $2k^{th}$ moment is a Meixner polynomial of degree $k$, as observed in \cite{Cunden19}.  Therefore the three-term recurrence and generating function can be explained as immediate consequences of the properties of the hypergeometric orthogonal polynomials.
\begin{rmk} \label{re:randHerm}
The generating function \eqref{eq:GUE_gf} can also be viewed as a randomisation of $n$,  the number of particles.  That is,  if we take the normalised moments $Q_k(n)$ and choose $(n-1)$ randomly according to a Negative Binomial distribution with parameters $(2,z)$,  we have that
\begin{equation}
\sum_{n \geq 1} Q_k(n) n (1-z)^2 z^{n-1} = (2k-1)!! \left( \frac{1+z}{1-z} \right)^k.
\end{equation}
\end{rmk}
Noting that the right hand side gives the $k$th moment of a Gaussian random variable with mean 0 and variance $(1+z)(1-z)^{-1}$, we therefore have the surprising fact that the one-point function of the GUE, when the number of particles is randomised in a certain way, is itself Normally distributed.
The equivalence of these probability densities can also be seen as a rewriting of Mehler's formula, 
\begin{equation}
\sum_{n \geq 0} \frac{ (\rho / 2)^n}{n!} H_n(x) H_n(y) = \frac{1}{\sqrt{1-\rho^2}} 
\exp{ \left( - \frac{ \rho^2(x^2 + y^2) - 2 \rho x y}{(1-\rho^2)} \right) }
\end{equation}
which is therefore providing another proof of Theorem \ref{th:HZ_gf}.
\subsection{Laguerre ensemble}
Consider an $(N+\alpha) \times N$ matrix $X = (X_{i,j})$ for some nonnegative integer $\alpha$, with independent identically distributed random complex Gaussian entries, so that the matrix $X$ is chosen according to the probability measure
\begin{equation}
P_N(X) \mathrm{d} X = \prod_{i,j=1}^N \frac{1}{\pi} e^{-|x_{ij}|^2} \mathrm{d} \Re x_{ij} \mathrm{d} \Im x_{ij},
\end{equation}
and define the $N \times N$ matrix
\begin{equation}
M = X^{\dagger} X.
\end{equation}
Then a matrix of the form $M$ is known as a complex Wishart matrix with parameter $\alpha$, and these matrices are said to belong the \emph{Laguerre Unitary Ensemble}, or LUE, as the measure is invariant under conjugation by unitary matrices. The distribution of the eigenvalues of $M$ is a probability measure on $[0,\infty)^N$
\begin{equation} \label{eq:Laguerre_ens}
\mathrm{d} Q_N(\lambda_1, \dots, \lambda_N) = \frac{1}{N! \prod_{i=0}^{N-1} (i+\alpha)! i!} \Delta_N (\lambda_1, \dots, \lambda_N)^2 \prod_{j=1}^N e^{-\lambda_j} \lambda_j^{\alpha} \mathrm{d} \lambda_j,
\end{equation}
which is an orthogonal polynomial ensemble \eqref{eq:dQ} with $\mu$ an Exponential measure with parameter $\alpha$.  The Laguerre polynomials are orthogonal with respect to a Gamma measure with parameters $(\alpha,1)$ on the positive half-line $[0,\infty)$,
\begin{equation}
\mathrm{d} \mu(x) = \frac{ x^{\alpha} e^{- x}}{\Gamma(\alpha+1)} \mathrm{d} x,
\end{equation}
and the $n^{th}$ normalised Laguerre polynomial $l_n(x; \alpha)$ can be written as \cite[9.12]{Koekoek10}
\begin{equation} \label{eq:Laguerre}
l_n(x; \alpha) = \frac{ (\alpha+1)_n}{\sqrt{n! \Gamma(n+\alpha+1)}} {}_1F_1 \left( \begin{matrix} -n,  \\ \alpha+1 \end{matrix} ; -x \right). 
\end{equation}
Hence the measure on the eigenvalues of the LUE can be called the \emph{Laguerre ensemble}. 
\begin{thm}[Hanlon-Stanley-Stembridge \cite{Hanlon92}] \label{th:Hanlon}
The expectation of a Schur polynomial under the Laguerre ensemble (with parameter $\alpha$) is given in terms of Schur polynomials.
\begin{equation}
\mathbb{E}_Q[s_{\lambda}(x_1, \dots, x_n)] = H_{\lambda} s_{\lambda}(1^n) s_{\lambda}(1^{n+\alpha}) 
\end{equation}
where $H_{\lambda}$ is given by the hook-length formula.
\end{thm}
Therefore by the symmetric functions formula relating power sums to Schur polynomials,  we have the following result immediately (since the power sums of the variables $(x_1, \dots,  x_n)$ are simply the traces of the powers of the matrix).
\begin{cor}[Hanlon-Stanley-Stembridge \cite{Hanlon92}]
The moments of the Laguerre ensemble one-point function are given by
\begin{equation}
\mathbb{E}[\text{tr} X_{n,n+\alpha}^k] = \sum_{\beta} H_{\beta} \chi^{\beta}(\lambda) s_{\beta}(1^n) s_{\beta}(1^{n+\alpha}) 
\end{equation}
where the $\chi^{\beta}(\lambda)$ are characters which can be computed by the Murnaghan-Nakayama rule.
\end{cor}
Let us denote the $k^{th}$ moment of the Laguerre ensemble by
\begin{equation}
Q_k(n; \alpha) := \frac{\mathbb{E}[\text{tr} X_{n, n+\alpha}^k ]}{n}.
\end{equation}
Then the following results about the moments are known.
\begin{thm}[Haagerup-Thorbj\o rnsen \cite{Haagerup03}]
The moments of the Laguerre ensemble $Q_k(n;\alpha)$ satisfy the following three-term recurrence in $k$,
\begin{equation} \label{eq:HT_rec}
(k+2) Q_{k+1}(n; \alpha) = (2k+1)(2n+ \alpha) Q_k(n ; \alpha) + (k-1)(k^2 - \alpha^2) Q_{k-1}(n;\alpha).
\end{equation} 
\end{thm}

See \cite{Koekoek10} for the definitions of the dual Hahn polynomials $R_n(x(x+\gamma+\delta+1); \gamma, \delta, N)$ and Hahn polynomials $S_n(x; \alpha, \beta, N)$.

\begin{thm}[Cunden-Mezzadri-O'Connell-Simm \cite{Cunden19}]
\label{th:HSS}
The $k^{th}$ moment of the Laguerre ensemble one-point function is given by the following hypergeometric series,  which can also be written as a dual Hahn polynomial of degree $n-1$, or as a Hahn polynomial of degree $k-1$.
\begin{align} \label{eq:LUE_mom}
Q_k(n; \alpha)
&= \frac{ (n+\alpha) (k+\alpha)!}{(1+\alpha)!} 
{}_3F_2 \left( \begin{matrix} 1-k ,  , 2+k, 1-n  \\ 2, 2+\alpha \end{matrix} ; 1 \right)
\nonumber \\
&= \frac{ (n+\alpha) (k+\alpha)!}{(1+\alpha)!}  R_{n-1} \left( (k-1)(k+2); 1, 1, -2-\alpha \right) \\
&= \frac{ (n+\alpha) (k+\alpha)!}{(1+\alpha)!} S_{k-1} \left( n-1; 1, 1, -2-\alpha \right).
\end{align}
\end{thm}

As shown in \cite{Cunden19}, the three-term recurrence \eqref{eq:HT_rec} can be explained as an immediate consequence of the recurrence in $k$ satisfied by the Hahn polynomials.  The dual Hahn polynomials satisfy several different generating functions, which therefore give rise to generating functions of the Laguerre ensemble moments. For example, they satisfy
\begin{equation} \label{eq:ctsHahn_gf}
\sum_{n \geq 0} R_n(x(x+\gamma+\delta+1) ; \gamma,\delta, N) \frac{ (\gamma+1)_n (-N)_n}{(-\delta-N)_n n!} t^n = (1-t)^x {}_2F_1 \left( \begin{matrix} x-N ,  x+\gamma+1  \\ -\delta-N \end{matrix} ; t \right).
\end{equation}
\begin{rmk} \label{re:randLag}
As in the Hermite ensemble, this generating function can be viewed as a randomisation of $n$, the number of particles.  If we choose $(n-1)$ randomly according to a Negative Binomial distribution with parameters $(2,z)$,  we have that
\begin{equation} \label{eq:LUE_gf}
\sum_{n \geq 1} Q_k(n;\alpha) n (1-z)^2 z^{n-1} = \frac{ (k+\alpha)!} {\alpha!} (1-z)^{k+1} {}_2F_1 \left( \begin{matrix} 1+k ,  1+k+\alpha  \\ 1+\alpha \end{matrix} ; z \right).
\end{equation}
\end{rmk}
\section{Preliminaries} \label{se:Pre}
There is a natural extension of the theory of hypergeometric series using $q$-series, to define what are known as basic hypergeometric series. We shall use the convention that the $q$-analogue of an integer $n$ is defined by 
\begin{equation}
[n]_q := \frac{ 1-q^n}{1-q}
\end{equation}
which clearly converges to $n$ as $q \to 1$.  Similarly we can define the $q$-analogue of the factorial
\begin{equation}
[n]_q! := \prod_{i=1}^{n} [i]_q = \frac{ (1-q^n)(1-q^{n-1}) \dots (1-q)}{(1-q)^n}.
\end{equation}
\subsection{Basic hypergeometric series}
The $q$-Pochhammer symbol is denoted
\begin{equation} \label{eq:qPoch}
(a; q)_n = (1-a)(1-aq) \dots (1-aq^{n-1}) = \prod_{i=0}^{n-1} (1-aq^i)
\end{equation}
and converges to the classical Pochhammer symbol when $a$ is a power of $q$, in the following way:
\begin{equation} \label{eq:qPoch_conv}
\frac{(q^j;q)_n }{(1-q)^n} = \frac{ [n+j-1]_q !}{[j-1]_q!} \to \frac{ (n+j-1)!}{(j-1)!} = (j)_n \quad \quad \text{as } \; q \to 1.
\end{equation}
The basic hypergeometric series is defined by
\begin{equation} \label{eq:basichyp}
 {}_r \phi_s \left( \begin{matrix} a_1, \dots , a_r \\ b_1 , \dots , b_s \end{matrix} ; q; z \right)=\sum_{i=0}^{\infty}\frac{(a_1; q)_i\cdots (a_r; q)_i}{(b_1; q)_i\cdots (b_s; q)_i}\frac{z^i}{(q; q)_i} \left( (-1)^i q^{i(i-1)/2} \right)^{1+s-r}.
\end{equation}
If one of the parameters $a_1, \dots, a_r$ is a non-positive power of $q$,  say $q^{-n}$ for some integer $n > 0$, then the sum terminates, since the $q$-Pochhammer symbol
\begin{equation}
(q^{-n}; q)_i = 0 \quad \quad \text{ for all } \quad i > n.
\end{equation}
Furthermore, if all the $a_1, \dots, a_r, b_1, \dots b_s$ are powers of $q$, then it is clear by \eqref{eq:qPoch_conv} that 
\begin{equation}
{}_r \phi_s \left( \begin{matrix} q^{c_1}, \dots , q^{c_r} \\ q^{d_1} , \dots , q^{d_s} \end{matrix} ; q; (q-1)^{1+s-r} z \right)  \to
{}_rF_s \left( \begin{matrix} c_1, \dots, c_r \\ d_1, \dots d_s \end{matrix}; z \right).
\end{equation}
\subsection{Jackson integrals} \label{se:Jackson}
The study of $q$-series and basic hypergeometric series is closely linked to that of $q$-calculus, or quantum calculus \cite{KacCheung},  which is often referred to as ``calculus without limits".  The $q$-analogue of the derivative of a function $f: \mathbb{R} \to \mathbb{R}$ is simply given by
\begin{equation}
\mathbb{D}_q f(x) := \frac{ f(x) - f(qx)}{ x - qx }
\end{equation}
which is known simply as the $q$-derivative.  It is clear that in the limit as $q \to 1$, this agrees with the classical derivative, and that the $q$-analogue of the integers play a similar role in $q$-calculus. For example
\begin{equation}
\mathbb{D}_q x^n = \frac{ x^n - (qx)^n}{ x - xq} = \frac{ (1-q^n)x^n}{ (1-q)x} = [n]_q x^{n-1}.
\end{equation}
Without going into details about uniqueness, it is clear that a good candidate for a $q$-anti-derivative is the so-called \emph{Jackson integral}:
\begin{equation}
\int f(x) \mathrm{d}_q x = \sum_{n \geq 0} (1-q) xq^n f(xq^n),
\end{equation}
from which we can define the \emph{definite Jackson integral}. For any function $f:[0,a] \to \mathbb{R}$, 
\begin{equation}
\int_0^a f(x) \mathrm{d}_q x := \sum_{n \geq 0} (1-q) aq^n f(aq^n).
\end{equation}
Intuitively,  this can be thought of formally as a type of infinite Riemann sum, where the $n^{th}$ rectangle has width~$(aq^n - aq^{n+1})$ and height $f(aq^n)$.  The definition extends to more general intervals as follows. For any function $f:[a,b] \to \mathbb{R}$, 
\begin{equation}
\int_a^b f(x) \mathrm{d}_q x = \int_0^b f(x) \mathrm{d}_q x - \int_0^a f(x) \mathrm{d}_q x
\end{equation}
which is well-defined for any $a < b$ by the convention that
\begin{equation}
\int_{-a}^0 f(x) \mathrm{d}_q x := \int_0^a f(-x) \mathrm{d}_q x  \quad \quad \text{ for } a > 0.
\end{equation}
The improper Jackson integral $\int_0^{\infty} f(x) \mathrm{d}_q x$ clearly cannot be defined by simply taking the limit $\lim_{a \to \infty} \int_0^a f(x) \mathrm{d}_q x$.  Instead, we make use of the fact that the positive half-line $[0,\infty)$ can be partitioned into intervals of the form $[q^{n+1}, q^n)$ for all $n \in \mathbb{Z}$. Hence for any function $f:[0,\infty] \to \mathbb{R}$, we can define
\begin{align}
\int_0^{\infty} f(x) \mathrm{d}_q x &= \sum_{n = -\infty}^{\infty} \int_{q^{n+1}}^{q^n} f(x) \mathrm{d}_q x 
\nonumber \\
&= \sum_{n = -\infty}^{\infty} (1-q)q^n f(q^n).
\end{align}
\section{Discrete $q$-Hermite ensemble} \label{se:qHerm}
We consider a rescaled version of the Discrete $q$-Hermite I polynomials \cite[14.28]{Koekoek10} defined as follows.  First, we define the reference measure $g(x;q)$ on the set $\{ \pm \nu q^j ; \; j \in \mathbb{Z}_+ \}$ where $\nu = \frac{1}{\sqrt{1-q}}$:
\begin{equation}
g( \{ x \} ; q) = \sum_{n = 0}^{\infty} \frac{ \nu q^n \mathbf{1}_{\nu q^n}(x) + \nu q^n \mathbf{1}_{- \nu q^n}(x)}{ (q;q)_n (-q;q)_n (-1;q)_{\infty}},
\end{equation}
which is defined as a Jackson integral for suitable test functions $f$
\begin{equation}
\int_{[-\nu,\nu]} f(x) \mathrm{d} g(x; q) =
\frac{\int_{-\nu }^{\nu} f\left( \frac{x}{\nu} \right) \left( \frac{ qx}{\nu};q \right)_{\infty} \left( -\frac{ qx}{\nu} ;q\right)_{\infty} \mathrm{d}_q x}
{(1-q) (q;q)_{\infty} (-q;q)_{\infty} (-1;q)_{\infty}} .
\end{equation}
This choice of $g$ is a natural $q$-analogue of the Gaussian distribution, and is usually referred to as the Gaussian $q$-distribution \cite{BozejkoYoshida, DiazTeruel05,DiazPariguan09}, as it converges to the Gaussian measure on $\mathbb{R}$ as $q \to 1$. Crucially, its odd moments are 0, and even moments are given by
\begin{equation}
\int x^{2k} \mathrm{d} g(x; q^2) = [2k-1]_{q^2} !!
\end{equation}
so this makes it a suitable $q$-analogue for considering moments. We thus have the rescaled $q$-Hermite polynomials which are orthogonal to the measure $g(x;q^2)$ defined by
\begin{equation}
\tilde{H}_n(x;q^2) = H_n\left( \frac{ x}{\nu} ; q^2 \right) = \frac{ q^{\frac{n(n-1)}{2}}}{ \sqrt{(q^2; q^2)_n}} \;
{}_2 \phi_1 \left( \begin{matrix} q^{-2n},  \frac{\nu}{x}  \\ 0 \end{matrix} ;q^2,  - \frac{ q^2 x}{ \nu } \right),
\end{equation}
and the $q$-Hermite orthogonal polynomial ensemble has normalised one-point function given by
\begin{equation}
\mathrm{d} \rho_N(x) = \frac{1}{N} \sum_{n=0}^{N-1} \tilde{H}_n(x; q^2)^2 \mathrm{d} g(x; q^2).
\end{equation}
We wish to consider the moments of this normalised one-point function. As in the classical Hermite ensemble, or GUE, the odd moments are 0 by symmetry, so we are only concerned with the even moments. We denote the $2k^{th}$ moment by
\begin{equation}
M(k,N) = \int_{- \nu}^{\nu} x^{2k} \mathrm{d} \rho_N(x) .
\end{equation}
We note that if $(X_1, \dots, X_N)$ is sampled from the $q$-Hermite ensemble, then 
\begin{equation}
M(k,N) = \mathbb{E}_Q \left[ \frac{1}{N} \sum_{i=1}^N X_i^{2k} \right].
\end{equation}
Morozov, Popolitov and Shakirov \cite{Morozov20} obtained the following remarkable formula for the generating function of the unnormalised moments. When $q \to 1$, this reduces to the Harer-Zagier generating function \eqref{eq:GUE_gf}.
\begin{thm}[Morozov-Popolitov-Shakirov (2020)]
The generating function of the $2k^{th}$ moment of the $q$-Hermite ensemble one-point function is
\begin{equation} \label{eq:qHerm_mom_gf}
\sum_{N \geq 0} \lambda^N \frac{  M(k,N) N }
{q^{k(1-k)} [2k-1]_{q^2} !!} =
 \frac{ q^k \lambda^{k+1} (-q^2 \lambda^{-1}; q^2)_k} {(q^{2k}-\lambda) (\lambda; q^2)_{k+1}}.
\end{equation}
\end{thm}
In \cite{Morozov20}, this is then used to obtain an explicit formula for the pure moments:
\begin{cor}[Morozov-Popolitov-Shakirov (2020)] \label{cor:Morozov_mom}
The $2k^{th}$ moment of the $q$-Hermite ensemble one-point function is given by
\begin{equation} \label{eq:Morozov1}
M(k,N) N = q^{-2Nk} \text{res}_{q^{2k}} + \sum_{a=0}^k q^{2Na} \text{res}_{q^{-2a}}
\end{equation}
where these ``residues" are given by
\begin{align} \label{eq:res1}
& \frac{\textit{res}_{q^{2k}}}{q^{k(1-k)} [2k-1]_{q^2}!!} = q^k \frac{ \prod_{n=1}^k (q^{2n} + q^{2k})}{ \prod_{n=0}^k (1-q^{2n+2k})}
 \\
& \label{eq:res2} \frac{\textit{res}_{q^{-2a}}}{q^{k(1-k)} [2k-1]_{q^2}!!} =  \frac{ q^{k-2a} q^{-2ka} \prod_{n=1}^k (1+q^{2n+2a}) }
{q^{-2a} (q^{2k+2a}-1) \prod_{n=0}^{a-1} (1-q^{2n-2a}) \prod_{n=a+1}^k (1-q^{2n-2a})} .
\end{align}
\end{cor}
We now show that both the moments of the $q$-Hermite ensemble, and the generating function, which can be rewritten as moments of a certain randomisation of the ensemble, can be written in terms of basic hypergeometric series. Furthermore, the randomised moments are themselves basic hypergeometric orthogonal polynomials, which therefore satisfy three-term recurrences.

\begin{thm} \label{th:qHerm1}
The $2k^{th}$ moment of the $q$-Hermite ensemble normalised one-point function is given by
\begin{align}
\frac{M(k,N)}{[2k-1]_{q^2}!!} 
&= \frac{q^{2k-2Nk-k^2} (-q^2 ; q^2 )_k}
{ (1-q^{2k}) (q^2;q^2)_k }
{}_3 \phi_2 \left( \begin{matrix} -q^{2k+2} , q^{2k}, q^{-2k}  \\ -q^2, q^{2k+2} \end{matrix} ; q^2, \; q^2 \right)
\nonumber \\
&+ \frac{q^{2k-k^2} (-q^2 ; q^2 )_k}{ (q^{2k} - 1) (q^2;q^2)_k }
{}_3 \phi_2 \left( \begin{matrix} -q^{2k+2} , q^{2k}, q^{-2k}  \\ -q^2, q^{2k+2} \end{matrix} ; q^2, \; q^{2(N+1)} \right) .
\end{align}
\end{thm}
\begin{proof}
This Theorem is a reinterpretation of Corollary \ref{cor:Morozov_mom}.  From \eqref{eq:Morozov1}, we have that
\begin{equation}
N M(k,N) = q^{-2Nk} \text{res}_{q^{2k}} + \sum_{a=0}^k q^{2Na} \text{res}_{q^{-2a}}.
\end{equation} 
We first consider the sum term alone, which given in $q$-Pochhammer symbols is
\begin{equation}
S_N:=q^{k(1-k)} [2k-1]_{q^2}!! \sum_{a=0}^k q^{2Na} \text{res}_{q^{-2a}} 
= \sum_{a=0}^k \frac{q^{2Na+k-2ka} (-q^{2a+2} ; q^2 )_k }{ (q^{2k+2a}-1) (q^{-2} ; q^{-2} )_a (q^2 ; q^2)_{k-a}}.
\end{equation}
Putting the $q$-Pochhammer symbols in terms of $q^2$ (rather than $q^{-2}$) and noting that
\begin{equation} \label{eq:q_rearr}
(q^2; q^2)_{k-a} = (-1)^a q^{a(a-1)} q^{-2ka} \frac{(q^2;q^2)_k}{ (q^{-2k} ; q^2)_a },
\end{equation}
we can rearrange this sum so that it is in the form of a basic hypergeometric function:
\begin{align}
S_N &= \sum_{a=0}^k \left( q^{2Na+k-2ka} \frac{ (-q^2 ; q^2 )_k (-q^{2k+2} ; q^2 )_a }{ (-q^2 ; q^2 )_a } \right) \div 
\nonumber \\ 
& \left( \left( (q^{2k} - 1) \frac{ (q^{2k+2} ; q^2)_a}{(q^{2k} ; q^2)_a}  \right)
\left( (-1)^a q^{-a(a+1)} (q^2 ; q^2)_a \right)
\left( (-1)^a q^{a(a-1)} q^{-2ka} \frac{(q^2;q^2)_k}{ (q^{-2k} ; q^2)_a } \right) \right) 
\nonumber \\
&= \frac{q^k (-q^2 ; q^2 )_k}{ (q^{2k} - 1) (q^2;q^2)_k }
\sum_{a=0}^k \frac{(-q^{2k+2} ; q^2 )_a  (q^{2k} ; q^2 )_a (q^{-2k} ; q^2 )_a }{(-q^2 ; q^2 )_a (q^{2k+2} ; q^2 )_a (q^2 ; q^2 )_a } q^{2Na + 2a}
\nonumber \\
&=  \frac{q^k (-q^2 ; q^2 )_k}{ (q^{2k} - 1) (q^2;q^2)_k }
{}_3 \phi_2 \left( \begin{matrix} -q^{2k+2} , q^{2k}, q^{-2k}  \\ -q^2, q^{2k+2} \end{matrix} ; q^2, \; q^{2(N+1)} \right) .
\end{align}
For the term that arises from the $\textit{res}_{q^{2k}}$, we note that it is clear that when $N=0$, the formula \eqref{eq:Morozov1} becomes
\begin{equation}
0 = q^0 \textit{res}_{q^{2k}} + S_0,
\end{equation}
and therefore $\textit{res}_{q^{2k}} = - S_0$, so 
\begin{equation}
N M(k,N) = -q^{-2Nk} S_0 + S_N,
\end{equation}
as required. 
\end{proof}
We remark that, in the above proof, the fact that
\begin{equation}
S_0 =-  \textit{res}_{q^{2k}}
\end{equation}
as defined in \eqref{eq:res1} can also be observed as a consequence from $q$-Saalschutz summation formula.

The generating function \eqref{eq:qHerm_mom_gf} in $N$ can clearly also be thought of as a randomisation of the number of particles $N$ in the ensemble. It is equivalent to choosing $(N-1)$ according to a Negative Binomial distribution with parameters $(2,\lambda)$ (that is, the same randomisation as in Remarks \ref{re:randHerm} and \ref{re:randLag}), after multiplying by $(1-\lambda)^2$:
\begin{equation}
(1-\lambda)^2 \sum_{N \geq 0} \lambda^N \frac{ N M(k,N) } {q^{k(1-k)} [2k-1]_{q^2} !!}
=\frac{ 1} {q^{k(1-k)} [2k-1]_{q^2} !!} \sum_{N \geq 1}  N M(k,N) \lambda^{N-1} (1-\lambda)^2 
\end{equation}
\begin{thm}
Let the $2k^{th}$ moment of the $q$-Hermite normalised one-point function, when $N-1$ is chosen randomly according to a Negative Binomial distribution with parameters $(2,\lambda)$ for some $0 < \lambda < 1$, be denoted
\begin{equation}
M_k(\lambda; q) := \sum_{N \geq 1} M(k,N) N \lambda^{N-1} (1-\lambda)^2.
\end{equation} 
Then this randomised moment is given by
\begin{equation}
M_k(\lambda; q) =
[2k-1]_{q^2}!!  \frac{ q^{-k} (-q^2 ; q^2 )_k (1-\lambda)}
{(q^2;q^2)_k (1-\lambda q^{-2k})}
{}_3 \phi_2 \left( \begin{matrix} -q^{2k+2}, q^{-2k}, \lambda  \\ -q^2, \lambda q^2 \end{matrix} ; q^2, \; q^2 \right).
\end{equation}
\end{thm}
\begin{proof}
First let us denote the prefactor 
\begin{equation}
A := [2k-1]_{q^2} !! \frac{q^k (-q^2 ; q^2 )_k (1- \lambda)}
{ (1-q^{2k}) (q^2;q^2)_k }
\end{equation}
hence by Theorem \ref{th:qHerm1} the randomised moment becomes
\begin{align}
& M_k(\lambda; q) = \sum_{N \geq 1} M(k,N) N \lambda^{N-1} (1-\lambda)^2
\nonumber \\
&= \sum_{N \geq 0} A \lambda^N \left\{ q^{-2Nk} 
{}_3 \phi_2 \left( \begin{matrix} -q^{2k+2} , q^{2k}, q^{-2m}  \\ -q^2, q^{2k+2} \end{matrix} ; q^2, \; q^2 \right)
- {}_3 \phi_2 \left( \begin{matrix} -q^{2k+2} , q^{2k}, q^{-2k}  \\ -q^2, q^{2k+2} \end{matrix} ; q^2, \; q^{2(N+1)} \right) \right\} 
\nonumber \\
&= A \sum_{N \geq 0} \lambda^N \sum_{a=0}^k
\frac{(-q^{2k+2} ; q^2 )_a  (q^{2k} ; q^2 )_a (q^{-2k} ; q^2 )_a }{(-q^2 ; q^2 )_a (q^{2k+2} ; q^2 )_a (q^2 ; q^2 )_a } q^{2a}
\Big\{ q^{-2Nk } - q^{2Na} \Big\} 
\nonumber \\
&= A \sum_{a=0}^k 
\frac{(-q^{2k+2} ; q^2 )_a  (q^{2k} ; q^2 )_a (q^{-2k} ; q^2 )_a }{(-q^2 ; q^2 )_a (q^{2k+2} ; q^2 )_a (q^2 ; q^2 )_a } q^{2a} 
\Big\{ \sum_{N \geq 0} \lambda^N q^{-2Nk} - \sum_{N \geq 0} \lambda^N q^{Na} \Big\} 
\nonumber \\
&= A \sum_{a=0}^k
\frac{(-q^{2k+2} ; q^2 )_a  (q^{2k} ; q^2 )_a (q^{-2k} ; q^2 )_a }{(-q^2 ; q^2 )_a (q^{2k+2} ; q^2 )_a (q^2 ; q^2 )_a } q^{2a} 
\Big\{ \frac{1}{(1-\lambda q^{-2k})} - \frac{1}{(1-\lambda q^{2a})} \Big\}.
\end{align}
Then we note that this term in the curly brackets becomes
\begin{equation}
\Big\{ \frac{1}{(1-\lambda q^{-2k})} - \frac{1}{(1-\lambda q^{2a})} \Big\} = 
\frac{ \lambda (q^{-2k} - q^{2a})} {(1-\lambda q^{-2k}) (1-\lambda q^{2a})}
=\lambda q^{-2k} 
\frac{ (1 - q^{2a+2k})} {(1-\lambda q^{-2k}) (1-\lambda q^{2a})}
\end{equation}
and we can rewrite this in terms of q-Pochhammer symbols as
\begin{equation}
\Big\{ \frac{1}{(1-\lambda q^{-2k})} - \frac{1}{(1-\lambda q^{2a})} \Big\} = 
\lambda q^{-2k}
\frac{ (q^{2k+2} ; q^2)_a (\lambda ; q^2)_a }
{ (q^{2k} ; q^2)_a (\lambda q^2 ; q^2)_a }
\frac{ (1-q^{2k}) } {(1-\lambda q^{-2k}) (1-\lambda)}.
\end{equation}
Therefore the randomised moment becomes
\begin{align}
M_k (\lambda; q)
&= A \sum_{a=0}^k
\frac{(-q^{2k+2} ; q^2 )_a  (q^{2k} ; q^2 )_a (q^{-2k} ; q^2 )_a }{(-q^2 ; q^2 )_a (q^{2k+2} ; q^2 )_a (q^2 ; q^2 )_a } q^{2a} 
\nonumber \\
& \times \Big\{ \lambda q^{-2k}
\frac{ (q^{2k+2} ; q^2)_a (\lambda ; q^2)_a }
{ (q^{2k} ; q^2)_a (\lambda q^2 ; q^2)_a }
\frac{ (1-q^{2k}) } {(1-\lambda q^{-2k})(1-\lambda)} \Big\} 
\nonumber \\
&= A \frac{ \lambda q^{-2k} (1-q^{2k}) } {(1-\lambda q^{-2k})(1-\lambda)}
\sum_{a=0}^k 
\frac{(-q^{2k+2} ; q^2 )_a (q^{-2k} ; q^2 )_a (\lambda ; q^2)_a}
{(-q^2 ; q^2 )_a (q^2 ; q^2 )_a (\lambda q^2 ; q^2)_a} q^{2a} 
\nonumber \\
&= A \frac{ \lambda } {(1-\lambda)} 
\frac{q^{-2k} (1-q^{2k})}{(1-\lambda q^{-2k})}
{}_3 \phi_2 \left( \begin{matrix} -q^{2k+2}, q^{-2k}, \lambda  \\ -q^2, \lambda q^2 \end{matrix} ; q^2, \; q^2 \right)
\end{align} 
which is exactly the basic hypergeometric series required.
\end{proof}
Now we note that the $q$-Hahn polynomials are defined in \cite[14.6]{Koekoek10} as the basic hypergeometric function
\begin{equation} \label{eq:qHahn}
Q_m(q^{-x}; \alpha, \beta, K | q) = {}_3 \phi_2 \left( \begin{matrix} q^{-m}, \alpha \beta q^{m+1}, q^{-x}  \\ \alpha q, q^{-K} \end{matrix} ; q, \; q \right)
\end{equation}
for $m=0, 1, \dots , K$ where $K$ is some fixed positive integer. Therefore moving to $q^2$, taking parameters $\alpha = -1$ and $\beta=1$, and setting $\lambda = q^{-2K-2}$ as the argument, we have
\begin{equation}
Q_k( q^{-2K-2}; -1, 1, K | q^2) = {}_3 \phi_2 \left( \begin{matrix} -q^{2k+2}, q^{-2k}, q^{-2K-2}  \\ -q^2, q^{-2K} \end{matrix} ; q^2, \; q^2 \right)
\end{equation}
Therefore we have the following result immediately.
\begin{cor}
The $2k^{th}$ moment of the $q$-Hermite normalised one-point function, when $N-1$ is chosen randomly according to a Negative Binomial distribution with parameters $(2,\lambda)$ for some $0 < \lambda < 1$, is given by
\begin{equation}
M_k(\lambda;q) =
\frac{ q^{-k} (-q^2 ; q^2 )_k}
{(q^2;q^2)_k (1- q^{-2k-2K-2})}
Q_k( q^{-2K-2}; -1, 1, K | q^2).
\end{equation}
\end{cor}
As an immediate result of the properties of the $q$-Hahn polynomials, the moments therefore satisfy a three-term recurrence.
\begin{cor}
Denoting the $2n^{th}$ randomised moment by $M_n(\lambda; q)$, the following three-term recurrence is satisfied.
\begin{equation}
A_n M_{n+1}(\lambda; q) + B_n M_n (\lambda; q) - C_n M_{n-1}(\lambda; q) = 0
\end{equation}
where 
\begin{align}
& A_n = \frac{(1-q^{2n-2K})(1 + q^{2n+2})^2}
{(1 + q^{4n+2})(1 + q^{4n+4})}  ,
\nonumber \\
& B_n =  \left( 1-q^{-2K-2} - \frac{(1-q^{2n-2K})(1 + q^{2n+2})^2}
{(1 + q^{4n+2})(1 + q^{4n+4})} + \frac{ q^{2n-2K} (1-q^{2n})^2 (1 + q^{2n+2K+2})}
{(1 + q^{4n})(1 + q^{4n+2})} \right)
\nonumber \\
& \times
\frac{q^{-1} (1+q^{2n+2})(1-q^{-2n-2K-2})}
{(1-q^{2n+2})(1-q^{-2n-2K-4})} ,
\nonumber \\
& C_n = \frac{ q^{2n-2K} (1-q^{2n})^2 (1 + q^{2n+2K+2})}
{(1 + q^{4n})(1 + q^{4n+2})}
\frac{q^{-2} (1+q^{2n+2})(1+q^{2n})(1-q^{-2n-2K})}
{(1-q^{2n+2})(1-q^{2n})(1-q^{-2n-2K-4})} 
\end{align}
\end{cor}
\begin{proof}
This comes as an immediate consequence of the following three-term recurrence satisfied by the $q$-Hahn polynomials:
\begin{equation} \label{eq:qHahn_rec}
a_n Q_{n+1}(q^{-x}) - \left( a_n + c_n + (1-q^{-x}) \right) Q_n(q^{-x}) + c_n Q_{n-1}(q^{-x})  = 0
\end{equation}
where we denote $Q_n(q^{-x}) =  Q_n(q^{-x}; \alpha, \beta, K; q)$, and
\begin{align}
& a_n = \frac{ (1-q^{n-K})(1- \alpha q^{n+1})(1- \alpha \beta q^{n+1})}
{ (1-\alpha \beta q^{2n+1})(1- \alpha \beta q^{2n+2}) },
\nonumber \\
& c_n = -\frac{ \alpha q^{n-K} (1-q^n)(1- \alpha \beta q^{n+K+1})(1- \beta q^n)}
{ (1- \alpha \beta q^{2n})(1 - \alpha \beta q^{2n+1})}.
\end{align}
\end{proof}

\section{Discrete $q$-Laguerre ensemble} \label{se:qLag}
We now consider the normalised $q$-Laguerre polynomials $L_n(x; \alpha,q)$ as defined in \cite[14.21]{Koekoek10}, which satisfy both a continuous and a discrete orthogonality relation.  The moments of the ensemble with respect to the continuous measure are explored in \cite{Tierz21}. Here we consider the discrete relation, where the polynomials are orthogonal with respect to the measure $\gamma^{\alpha}$ on the interval $[0,\infty)$ defined, for appropriate test functions $f(x)$, by
\begin{equation}
\int_{[0,\infty)} f(x) \mathrm{d} \gamma^{\alpha}(x) =
\frac{2 (q^{\alpha+1};q)_{\infty} (-q;q)_{\infty}^2}
{(1-q)(-q^{\alpha+1};q)_{\infty} (-q^{-\alpha};q)_{\infty} (q;q)_{\infty}^2}
\int_0^{\infty} \frac{ f(x) x^{\alpha}} {(-x; q)_{\infty}} \mathrm{d}_q x,
\end{equation}
where the integral on the right hand side is a Jackson integral.  This is an atomic measure, supported on $\{ q^{\pm j}; \; j \in \mathbb{Z}_+ \}$, with
\begin{equation} \label{eq:qLaguerre_meas}
\gamma^{\alpha}( \{ x \} ) =
\frac{2 (q^{\alpha+1};q)_{\infty} (-q;q)_{\infty}}
{(-q^{\alpha+1};q)_{\infty} (-q^{-\alpha};q)_{\infty} (q;q)_{\infty}^2}
 \sum_{n = - \infty}^{\infty}  q^{n(\alpha+1)} (-q;q)_n \mathbf{1}_{q^n} (x)
\end{equation}
The $n^{th}$ normalised $q$-Laguerre polynomial $L_n(x; \alpha, q)$ can be written as~\cite[14.21]{Koekoek10}
\begin{align} \label{eq:qLaguerre}
L_n(x;q) &= 
\sqrt{ \frac{ q^n}{(q;q)_n (q^{\alpha+1};q)_n}}
{}_2 \phi_1 \left( \begin{matrix} q^{-n},  -x  \\ 0 \end{matrix} ;q,  q^{n+\alpha+1} \right) 
\nonumber \\
&= \sqrt{ \frac{(q^{\alpha+1};q)_n q^n}{(q;q)_n}}
{}_1 \phi_1 \left( \begin{matrix} q^{-n} \\ q^{\alpha+1} \end{matrix} ;q,  -q^{n+\alpha+1}x \right).
\end{align}

The corresponding orthogonal polynomial ensemble \eqref{eq:dQ} is the probability measure on $[0,\infty)^N$
\begin{equation}
\mathrm{d} Q(x) = \frac{1}{Z_N} \Delta_N(\mathbf{x})^2 \prod_{j=1}^N \mathrm{d} \gamma^{\alpha} (x_j),
\end{equation}
and has a normalised one-point function given by
\begin{equation}
\mathrm{d} \rho_N^{\alpha} (x) = \frac{1}{N} \sum_{n=0}^{N-1} L_n(x; \alpha, q)^2 \mathrm{d} \gamma^{\alpha} (x).
\end{equation} 
We denote the $k^{th}$ moment of the normalised one-point function of the $q$-Laguerre ensemble
\begin{equation}
M_{\alpha}(k,N) := \int_0^{\infty} x^k \mathrm{d} \rho_N^{\alpha} (x),
\end{equation}
and we can calculate these moments using a similar method as in \cite{Morozov20} following the proof of Theorem \ref{th:HSS}. Unlike in the $q$-Hermite case, this is found to be a single basic hypergeometric series.  Our starting point is the following $q$-analogue of Theorem \ref{th:Hanlon}, due to Morozov, Popolitov and Shakirov \cite{MPSqt}.
\begin{lem}[Morozov-Popolitov-Shakirov (2018)] \label{le:qLag_schur}
The expectation of a Schur polynomial under the $q$-Laguerre ensemble is given by
\begin{equation}
\mathbb{E}[s_{\lambda} (X_1, \dots, X_N)] = 
\frac{ q^{k(2-2N-\alpha)}}{ (1-q^2)^k}
\frac{ s_{\lambda}(1,q^2, \dots, q^{2N-2}) s_{\lambda}(1,q^2, \dots, q^{2N+2 \alpha-2})}
{ s_{\lambda}(1,q^2, q^4, \dots)}.
\end{equation}
\end{lem}
We also recall formulas for specialisations of the Schur function which may be found, for example, in \cite[Theorem 7.21.2,  Corollary 7.21.3]{Stanley}.
\begin{lem} \label{le:ps_Schur}
The principal specialisations of the Schur function $s_{\lambda}(1,q^2, q^4, \dots)$ and $s_{\lambda}(1,q^2, \dots, q^{2N-2})$ can be written as
\begin{align} \label{eq:psSchur_proof}
& s_{\lambda}(1,q^2, q^4, \dots) = \frac{ q^{2 b(\lambda)} }{ [H_{\lambda}]_{q^2} (1-q^2)^k} \\
& s_{\lambda}(1,q^2, \dots, q^{2N-2}) = q^{2b(\lambda) } 
\prod_{(i,j) \in \lambda} \frac{ (1-q^{2(N+j-i)})}{(1-q^{2h_{\lambda}(i,j)})}
\end{align}
where $b(\lambda) = \sum (i-1) \lambda_i$, and $h_{\lambda}(i,j)$ is the hook length of the box $(i,j)$ in the Young diagram of the partition $\lambda$.
\end{lem}
We can therefore calculate the moments of the $q$-Laguerre ensemble as follows.
\begin{thm} \label{th:qLag}
The $k^{th}$ moment of the $q$-Laguerre normalised one-point function is given by the basic hypergeometric series
\begin{equation} \label{eq:qLag_mom}
M_{\alpha}(k,N) = 
\frac{ q^{k(2-2N-\alpha)} }{N  (1-q^2)^k }
\frac{ (q^{2N}; q^2)_k (q^{2N+2 \alpha}; q^2)_k}{(q^2; q^2)_k} {}_3 \phi_2 \left( \begin{matrix} 
q^{2-2k} , q^{2-2N}, q^{2-2N-2\alpha}  \\ q^{2-2N-2k}, q^{2-2N-2\alpha-2k} \end{matrix} ; q^2, \; q^{-2k} \right).
\end{equation}
\end{thm}
\begin{proof}
We first recall the expansion of power sums in terms of Schur polynomials:
\begin{equation}
p_{\rho} = \sum_{\lambda} \chi_{\rho}^{\lambda} s_{\lambda}.
\end{equation}
Using the Murnaghan-Nakayama rule \cite[7.17]{Stanley}, it is clear that the only non-zero characters for a partition $(k)$, $\chi_{(k)}^{\lambda}$, come from the partitions with no 2x2 boxes, i.e.\ hook-shaped partitions $\lambda=(k-l, 1^l)$, which are given by $\chi_{(k)}^{(k-l, 1^l)} = (-1)^l$. Hence
\begin{equation}
p_{(k)} = \sum_{\lambda} \chi_{(k)}^{\lambda} s_{\lambda} = \sum_{l=0}^{k-1} (-1)^l s_{(k-l, 1^l)}.
\end{equation}
Therefore by linearity of expectation we have
\begin{equation}
M_{\alpha}(k,N) = 
\frac{1}{N} \mathbb{E}_Q \left[ p_{(k)}(X_1, \dots, X_N) \right] = 
\frac{1}{N} \sum_{l=0}^{k-1} (-1)^l \mathbb{E}_Q[s_{(k-l, 1^l)} (X_1, \dots, X_N)],
\end{equation}
and thus by Lemma \ref{le:qLag_schur}
\begin{equation} \label{eq:qLag_proof}
M_{\alpha}(k,N) = \sum_{l=0}^{k-1} (-1)^l 
\frac{ q^{k(2-2N-\alpha)}}{ N (1-q^2)^k}
\frac{ s_{(k-l, 1^l)}(1,q^2, \dots, q^{2N-2}) s_{(k-l, 1^l)}(1,q^2, \dots, q^{2N+2 \alpha -2})}
{ s_{(k-l, 1^l)}(1,q^2, q^4, \dots)}.
\end{equation}
Noting that the $q$-analogue of the product of the hook lengths can be written as
\begin{equation}
[H_{\lambda}]_{q^2} = \prod_{(i,j) \in \lambda} \frac{ 1-q^{2h_{\lambda}(i,j)}}{1-q^2},
\end{equation} 
using Lemma \ref{le:ps_Schur}, we can give the principal specialisations of the Schur functions for a hook-shaped partition $\lambda = (k-l,1^l)$ explicitly in terms of $q$-Pochhammer symbols.  In this case,
\begin{equation}
b((k-l, 1^l)) = \sum_{i=1}^{l+1} (i-1) \lambda_i = \sum_{i=2}^{l+1} (i-1) = 
\frac{l(l+1)}{2},
\end{equation}
and so
\begin{align}
s_{(k-l, 1^l)}(1,q^2, \dots, q^{2N-2}) 
&= q^{l(l+1)} \frac{ \prod_{(i,j) \in \lambda} (1 - q^{2(N+j-i)})}{[H_{\lambda}]_{q^2} (1-q^2)^k}  \\
&= q^{l(l+1)} \frac{ (q^{2N-2l} ; q^2)_k }{ (1-q^{2k}) (q^2; q^2)_l (q^2; q^2)_{k-l-1}}.
\end{align}
Therefore substituting this into equation \eqref{eq:qLag_proof} (and noting the cancellation of the terms in \eqref{eq:psSchur_proof}), we have the moment formula
\begin{align} \label{eq:qLag_sum}
M_{\alpha}(k,N)
&= \sum_{l=0}^{k-1} (-1)^l q^{k(2-2N-\alpha) + l(l+1)} 
\frac{ (q^{2N-2l}; q^2)_k (q^{2N+2\alpha-2l} ; q^2)_k} {N (1-q^2)^k (1-q^{2k}) (q^2; q^2)_l (q^2; q^2)_{k-l-1}} 
\nonumber \\
&= \frac{ q^{k(2-2N-\alpha)} }{ N (1-q^2)^k (1-q^{2k})}
\sum_{l=0}^{k-1} (-1)^l q^{l(l+1)} \frac{ (q^{2N-2l}; q^2)_k (q^{2N+2\alpha-2l} ; q^2)_k}{(q^2; q^2)_l (q^2; q^2)_{k-l-1}} 
\end{align}
and these $q$-Pochhammer symbols can be transformed into the form necessary for a basic hypergeometric series as follows: firstly, by \eqref{eq:q_rearr} we have
\begin{equation}
(q^2; q^2)_{k-l-1} = (-1)^l q^{-2kl} q^{ l(l+1)}  \frac{ (q^2; q^2)_{k-1}} {(q^{2-2k}; q^2)_l}.
\end{equation}
Furthermore, we can rewrite 
\begin{align}
(q^{2N-2l}; q^2)_k &=
(1-q^{2N-2l}) (1-q^{2N-2l+2}) \dots (1-q^{2N-2l+2k-2}) 
\nonumber \\
&= (1-q^{2N})(1-q^{2N+2}) \dots (1-q^{2N+2k-2})
\frac{ (1-q^{2N-2}) \dots (1-q^{2N-2l})} {(1-q^{2N+2k-2}) \dots (1-q^{2N+2k-2l})}
\nonumber \\
&= (q^{2N}; q^2)_k
\frac{ (q^{2N-2}; q^{-2})_l}{ (q^{2N+2k-2} ; q^{-2})_l}
\nonumber \\
&= (q^{2N}; q^2)_k
\frac{ (-1)^l q^{2Nl - l(l+1)} (q^{2-2N}; q^2)_l} { (-1)^l q^{2Nl + 2kl - l(l+1)} (q^{2-2N-2k}; q^2)_l} 
\nonumber \\
&= (q^{2N}; q^2)_k \frac{ (q^{2-2N}; q^2)_l}{q^{2kl} (q^{2-2N-2k}; q^2)_l}, 
\end{align}
and similarly
\begin{equation}
(q^{2N+2\alpha-2l}; q^2)_k = (q^{2N+2 \alpha}; q^2)_k \frac{ (q^{2-2N-2\alpha}; q^2)_l}{q^{2kl} (q^{2-2N- 2 \alpha -2k}; q^2)_l}.
\end{equation}
Hence putting these 3 terms together and multiplying by $(-1)^l q^{l(l+1)}$, we have that the summand in \eqref{eq:qLag_sum} is equal to
\begin{equation}
\frac{ (q^{2N}; q^2)_k (q^{2N+2\alpha}; q^2)_k}{(q^2; q^2)_{k-1} } 
\frac{ (q^{2-2k}; q^2)_l (q^{2-2N}; q^2)_l (q^{2-2N-2\alpha}; q^2)_l}{q^{2kl} (q^{2-2N-2k}; q^2)_l (q^{2-2N-2\alpha-2k}; q^2)_l}.
\end{equation}
Therefore the $k^{th}$ moment of the $q$-Laguerre normalised one-point function is
\begin{multline}
M_{\alpha}(k,N) = 
\frac{ q^{k(2-2N-\alpha)} }{ (1-q^2)^k (1-q^{2k})}
\frac{ (q^{2N}; q^2)_k (q^{2N+2\alpha}; q^2)_k}{(q^2; q^2)_{k-1} } 
\nonumber \\
\times \sum_{l=0}^{k-1} q^{-2kl} 
\frac{ (q^{2-2k}; q^2)_l (q^{2-2N}; q^2)_l (q^{2-2N-2\alpha}; q^2)_l}{(q^2; q^2)_l (q^{2-2N-2k}; q^2)_l (q^{2-2N-2\alpha-2k}; q^2)_l}
\end{multline}
which is the basic hypergeometric series required.
\end{proof}
We note that this moment formula \eqref{eq:qLag_mom} is a $q$-analogue of the classical LUE moment formula \eqref{eq:LUE_mom}, which can be recovered by taking $q \to 1$. As in the $q$-Hermite ensemble, we find that when we choose the number of particles randomly according to a negative binomial distribution, the moments of this randomised ensemble are in fact basic hypergeometric orthogonal polynomials. 
\begin{thm}
Let the $k^{th}$ moment of the $q$-Laguerre normalised one-point function, when $N-1$ is chosen randomly according to a Negative Binomial distribution with parameters $(2,z)$ for some $0 < z < 1$, be denoted
\begin{equation}
M_k(z; \alpha, q) := \sum_{N \geq 1} M_{\alpha}(k,N) N z^{N-1} (1-z)^2.
\end{equation}
Then this randomised moment is given by
\begin{equation} \label{eq:qLag_thm}
M_k(z; \alpha, q) =
 \frac{ [k+\alpha]_{q^2}! }{ [\alpha]_{q^2}! } (z; q^{-2})_k q^{-\alpha k}  (1-z)
{}_2 \phi_1 \left( \begin{matrix} q^{2k+2}, q^{2k+2+2 \alpha}  \\ q^{2 \alpha + 2} \end{matrix} ; q^2, \; zq^{-2k} \right).
\end{equation}
\end{thm}
\begin{proof}
We recall by Theorem \ref{th:qLag} that
\begin{equation}
M_{\alpha}(k,N) = 
\frac{ q^{k(2-2N-\alpha)} }{ N(1-q^2)^k }
\frac{ (q^{2N}; q^2)_k (q^{2N+2 \alpha}; q^2)_k}{(q^2; q^2)_k} {}_3 \phi_2 \left( \begin{matrix} 
q^{2-2k} , q^{2-2N}, q^{2-2N-2 \alpha}  \\ q^{2-2N-2k}, q^{2-2N-2k-2 \alpha} \end{matrix} ; q^2, \; q^{-2k} \right).
\end{equation}
Randomising the number of particles, we therefore have
\begin{align} \label{eq:qLag_rand}
& M_k(z;\alpha,q) = \sum_{N \geq 1} (1-z)^2 z^{N-1} \frac{ q^{k(2-2N-\alpha)} }{(1-q^2)^k }
\frac{ (q^{2N}; q^2)_k (q^{2N+2 \alpha}; q^2)_k}{(q^2; q^2)_k}
\nonumber \\
& \times {}_3 \phi_2 \left( \begin{matrix} 
q^{2-2k} , q^{2-2N}, q^{2-2N-2 \alpha}  \\ q^{2-2N-2k}, q^{2-2N-2k-2 \alpha} \end{matrix} ; q^2, \; q^{-2k} \right)
\nonumber \\
&= \sum_{N \geq 1} (1-z)^2 q^{- \alpha k} z^{N-1} 
\frac{ q^{2k(1-N)}}{ (1-q^2)^k} \frac{ (q^{2N}; q^2)_k (q^{2N+ 2 \alpha}; q^2)_k}{(q^2; q^2)_k}
\nonumber \\
& \times
\sum_{n=0}^{(k-1) \wedge (N-1)} \frac{ (q^{2-2k}; q^2)_n (q^{2-2N}; q^2)_n (q^{2-2N-2 \alpha}; q^2)_n} {(q^{2-2N-2k}; q^2)_n (q^{2-2N-2k-2 \alpha}; q^2)_n (q^2; q^2)_n} q^{-2kn}
\nonumber \\
&= \frac{ (1-z)^2}{(1-q^2)^k (q^2; q^2)_k} \sum_{n=0}^{k-1} q^{-2kn} 
\frac{ (q^{2-2k}; q^2)_n}{(q^2; q^2)_n} 
\nonumber \\
& \times
\sum_{N \geq n+1} 
\left( \frac{ (q^{2-2N}; q^2)_n}{(q^{2-2N-2k}; q^2)_n} (q^{2N}; q^2)_k \right)
\left( \frac{ (q^{2-2N-2 \alpha}; q^2)_n}{(q^{2-2N-2k-2 \alpha}; q^2)_n} (q^{2N+2 \alpha}; q^2)_k \right) q^{2k(1-N)} z^{N-1}.
\end{align}
Rearranging the $q$-Pochhammer symbols, we find that
\begin{equation}
\left( \frac{ (q^{2-2N}; q^2)_n}{(q^{2-2N-2k}; q^2)_n} (q^{2N}; q^2)_k \right) =
 q^{2kn} \frac{ (q^2; q^2)_k (q^{2k+2}; q^2)_{N-n-1} }{ (q^2; q^2)_{N-n-1}}
\end{equation}
and by the same logic
\begin{equation}
\left( \frac{ (q^{2-2N-2 \alpha}; q^2)_n}{(q^{2-2N-2k-2 \alpha}; q^2)_n} (q^{2N+2 \alpha}; q^2)_k \right) =
 q^{2kn} \frac{ (q^2; q^2)_{k+\alpha} (q^{2k+2+2 \alpha}; q^2)_{N-n-1} }{ (q^2; q^2)_{\alpha} (q^{2 \alpha + 2}; q^2)_{N-n-1}}.
\end{equation}
Therefore the $N$ sum in \eqref{eq:qLag_rand} becomes
\begin{align}
& \sum_{N \geq n+1} q^{4kn} 
\left( \frac{ (q^2; q^2)_k (q^{2k+2}; q^2)_{N-n-1}}{(q^2; q^2)_{N-n-1}} \right)
\left( \frac{ (q^2; q^2)_{k+\alpha} (q^{2k+2+2 \alpha}; q^2)_{N-n-1} }{ (q^2; q^2)_{\alpha} (q^{2 \alpha+ 2}; q^2)_{N-n-1}} \right)  (zq^{-2k})^{N-1}
\nonumber \\
&=  q^{4kn} (q^2; q^2)_k (q^2; q^2)_{k + \alpha} \sum_{j \geq 0} \frac{(q^{2k+2}; q^2)_j (q^{2k+2+2 \alpha}; q^2)_j }{(q^{2 \alpha+2}; q^2)_j (q^2; q^2)_j} (zq^{-2k})^{j+n} 
\nonumber \\
&=  z^n q^{2kn} (q^2; q^2)_k (q^2; q^2)_{k + \alpha} \sum_{j \geq 0} \frac{(q^{2k+2}; q^2)_j (q^{2k+2+2 \alpha}; q^2)_j }{(q^{2 \alpha+2}; q^2)_j (q^2; q^2)_j} (zq^{-2k})^j,
\end{align}
thus giving the randomised moment
\begin{equation}
M_k(z;\alpha,q) = (1-z)^2 \frac{ (q^2; q^2)_{k+\alpha} q^{-\alpha k}}{(1-q^2)^k (q^2; q^2)_{\alpha}}  
{}_1 \phi_0 \left( \begin{matrix} q^{2-2k}  \\ - \end{matrix} ; q^2, \; z \right) 
{}_2 \phi_1 \left( \begin{matrix} q^{2k+2}, q^{2k+2+2 \alpha}  \\ q^{2 \alpha + 2} \end{matrix} ; q^2, \; zq^{-2k} \right).
\end{equation}
Hence by the $q$-binomial theorem, which can be stated as
\begin{equation}
{}_1 \phi_0 \left( \begin{matrix} a  \\ - \end{matrix} ; q \; z \right) = \frac{ (az;q)_{\infty}} {(z;q)_{\infty}},
\end{equation}
and the definition of the $q$-integers, we have the required formula \eqref{eq:qLag_thm}.
\end{proof}
We again note that this randomised moment formula \eqref{eq:qLag_thm} is a $q$-analogue of the LUE generating function \eqref{eq:LUE_gf}, which can be recovered by taking $q \to 1$. This basic hypergeometric series can be rewritten using some transformations  due to Heine (which can be found in \cite[1.13.3]{Koekoek10} for example) and Jackson \cite[1.13.17]{Koekoek10}:
\begin{align}
& \label{eq:tr_Heine} {}_2 \phi_1 \left( \begin{matrix} a,b  \\ c \end{matrix} ; q, \; z \right) = 
\frac{ (abc^{-1}z; q)_{\infty}}{ z; q)_{\infty}}
{}_2 \phi_1 \left( \begin{matrix} a^{-1} c,b^{-1} c  \\ c \end{matrix} ; q, \; abzc^{-1} \right),  \\
& \label{eq:tr_Jackson} {}_2 \phi_1 \left( \begin{matrix} q^{-n} ,  b  \\ c \end{matrix} ; q, \; z \right) =
\frac{ (q^{-n} bc^{-1}z; q)_{\infty}}{ bc^{-1} z ; q)_{\infty}}
{}_3 \phi_2 \left( \begin{matrix} q^{-n} ,  b^{-1}c  \\ 0 \end{matrix} ; q, \; q \right).
\end{align}
We finally note the definition of the unnormalised $n^{th}$ Big $q$-Jacobi polynomial \cite[14.5]{Koekoek10}:
\begin{equation} \label{eq:qJac}
B_n(x; a,b,c; q) = {}_3 \phi_2 \left( \begin{matrix} q^{-n}, abq^{n+1}, x \\ aq, cq \end{matrix} ; q, q \right),
\end{equation}
so choosing the parameters appropriately, we have that
\begin{equation}
B_k(0; q^{2 \alpha}, q^{-2 \alpha}, z^{-1}; q^2) = 
{}_3 \phi_2 \left( \begin{matrix} q^{-2k}, q^{2k+2}, 0  \\ q^{2 \alpha+2},  z^{-1} q^2 \end{matrix} ; q^2, \; q^2 \right).
\end{equation}
\begin{cor}
The randomised $q$-Laguerre moment $M_k(z;\alpha,q)$ can be written in terms of a Big $q$-Jacobi polynomial:
\begin{equation}
M_k(z;\alpha,q) = \frac{ [k+\alpha]_{q^2}! }{ [\alpha]_{q^2}! }
\frac{ (z; q^{-2})_k }{ (zq^2 ; q^2)_k }q^{-\alpha k} 
B_k(0; q^{2 \alpha}, q^{-2 \alpha}, z^{-1}; q^2).
\end{equation}
\end{cor}
\begin{proof}
Using the transformations of Heine and Jackson \eqref{eq:tr_Heine},\eqref{eq:tr_Jackson},  we can rewrite $M_k(z;\alpha,q)$ as a ${}_3 \phi_2$ basic hypergeometric series:
\begin{align}
& M_k(z;\alpha,q) =
\frac{ [k+\alpha]_{q^2}! }{ [\alpha]_{q^2}! } (z; q^{-2})_k q^{-\alpha k}  (1-z)
{}_2 \phi_1 \left( \begin{matrix} q^{2k+2}, q^{2k+2+2 \alpha}  \\ q^{2 \alpha + 2} \end{matrix} ; q^2, \; zq^{-2k} \right)
\nonumber \\
&= \frac{ [k+\alpha]_{q^2}! }{ [\alpha]_{q^2}! } (z; q^{-2})_k q^{-\alpha k}  (1-z)
\frac{ (zq^{2k+2}; q^2)_{\infty}} {(zq^{-2k}; q^2)_{\infty} }
{}_2 \phi_1 \left( \begin{matrix} q^{-2k+2 \alpha}, q^{-2k}  \\ q^{2 \alpha+2} \end{matrix} ; q^2, \; zq^{2k+2} \right)
\nonumber \\
&= \frac{ [k+\alpha]_{q^2}! }{ [\alpha]_{q^2}! } (z; q^{-2})_k q^{-\alpha k}  (1-z)
\frac{ (zq^{2k+2}; q^2)_{\infty}} {(zq^{-2k}; q^2)_{\infty} }
\frac{ (zq^{-2k}; q^2)_{\infty}} {(z; q^2)_{\infty} }
{}_3 \phi_2 \left( \begin{matrix} q^{-2k}, q^{2k+2}, 0  \\ q^{2 \alpha+2},  z^{-1} q^2 \end{matrix} ; q^2, \; q^2 \right),
\end{align}
which by the definition \eqref{eq:qJac} is exactly the form of the Big $q$-Jacobi polynomial required.
\end{proof}
By the general properties of the orthogonal polynomials, the moments of this randomised ensemble therefore immediately satisfy a three-term recurrence.
\begin{cor}
Denoting the $k^{th}$ randomised moment by $M_k(z;\alpha,q)$, the following three-term recurrence is satisfied.
\begin{equation}
A_k M_{k+1}(z; \alpha,q) + D_k M_k(z;\alpha,q) - C_k M_{k+1}(z;\alpha,q) = 0
\end{equation}
where
\begin{align}
& A_k = \frac{(1-q^{2k+2} z^{-1}) }
{ (1-q^{4k+2}) (1+q^{2k+2} )},
\nonumber \\
& D_k = \frac{ (1-zq^{-2k} )q^{- \alpha}} { (1-q^2) (1-zq^{2k+2})}
 \left(1 - (1-q^{2k+2 \alpha + 2}) A_k
- \frac{ q^{2k+2 \alpha + 2} (1-zq^{2k}) (1-q^{2k+2 \alpha}) }
{ z (1-q^{4k+2}) (1+q^{2k} )}  \right) ,
\nonumber \\
& C_k = \frac{ (1-q^{2k+2 \alpha}) (1-zq^{-2k}) (1-zq^{2-2k}) } 
{ (1-q^2)^2 (1-zq^{2k+2}) (1-zq^{2k}) }
\left( \frac{ q^{2k+2} (1-zq^{2k}) (1-q^{2k+2 \alpha}) }
{ z(1-q^{4k+2}) (1+q^{2k} )}  \right)
\end{align}
\end{cor}
\begin{proof}
This follows as an immediate consequence of the three-term recurrence satisfied by the Big $q$-Jacobi polynomials:
\begin{equation} \label{eq:qJac_rec}
a_n B_{n+1}(x) - \left( a_n + c_n + (x-1) \right) B_n(x) + c_n B_{n-1}(x)  = 0
\end{equation}
where we denote $B_n(x) = B_n(x; a,b,c; q)$, and
\begin{align}
& a_n = \frac{ (1-aq^{n+1})(1- abq^{n+1})(1-c q^{n+1})}
{ (1-ab q^{2n+1})(1- ab q^{2n+2}) },
\nonumber \\
& c_n = -\frac{ ac q^{n+1} (1-q^n)(1- abc^{-1} q^n)(1- b q^n)}
{ (1- ab q^{2n})(1 - ab q^{2n+1})}.
\end{align}
\end{proof}

\end{document}